\documentclass[12pt]{amsart}
\usepackage{fullpage}
\usepackage{graphicx} 
\usepackage{amsfonts,amsmath,amsthm,amssymb}
\usepackage{enumerate}
\usepackage[dvipsnames]{xcolor} 
\usepackage{tikz}
\usetikzlibrary{matrix}
\usepackage[colorlinks]{hyperref}

\usepackage{setspace}
\usepackage{lineno} 

\newtheorem{theorem}{Theorem}
\newtheorem{lemma}[theorem]{Lemma}
\newtheorem{definition}[theorem]{Definition}
\newtheorem{corollary}[theorem]{Corollary}

\newtheorem{case}{Case}

\newtheorem{example}[theorem]{Example}

\usepackage{mathtools}
\DeclarePairedDelimiter{\floor}{\lfloor}{\rfloor}
\DeclarePairedDelimiter{\ceil}{\lceil}{\rceil}

\title{On the $\delta$-chromatic numbers of the Cartesian products of graphs}

\author{Wipawee Tangjai}
\address{Department of Mathematics, Faculty of Science, Mahasarakham University, Maha Sarakham 44150, Thailand}
\email{wipawee.t@msu.ac.th}

\author{Witsarut Pho-on}
\address{Department of Mathematics, Faculty of Science, Srinakharinwirot University, Sukhumvit 23, 10110 Bangkok, Thailand}
\email{witsarut@g.swu.ac.th}

\author{Panupong Vichitkunakorn*}
\address{Division of Computational Science, Faculty of Science, Prince of Songkla University, Songkla 90110, Thailand}
\email{panupong.v@psu.ac.th}

\begin{document}
\maketitle
\begin{abstract}
In this work, we study the $\delta$-chromatic number of a graph which is the chromatic number of the $\delta$-complement of a graph. 
We give a structure of the $\delta$-complements and sharp bounds on the $\delta$-chromatic numbers of the Cartesian products of graphs.
Furthermore, we compute the $\delta$-chromatic numbers of various classes of Cartesian product graphs, including the Cartesian products between cycles, paths, and stars.

\end{abstract}


\section{Introduction}
The concept of $\delta$-complement was introduced in 2022 \cite{math10081203}. Their research focused on exploring various intriguing characteristics of these graphs, including properties like $\delta$-self-complementary, adjacency, and hamiltonicity.
In 2023, Vichitkunakorn et al. \cite{Vichitkunakorn2023} introduced the term $\delta$-chromatic number of a graph $G$ which refers to the chromatic number of the $\delta$-complement of $G$.
They established a Nordhaus-Gaddum bound type relation between the chromatic number and the $\delta$-chromatic number across various parameters: the clique number, the number of vertices and the degrees of vertices.
The given bounds are sharp and the classes of graphs satisfying those bounds are given
\cite{Vichitkunakorn2023}.
In this study, we present a more detailed outcome concerning the $\delta$-chromatic number of the Cartesian product of graphs.

In 1957, Sabidussi \cite{sabidussi_1957} showed that the chromatic number of the Cartesian product graphs is equal to the maximum chromatic number between such two graphs.
A lot of subsequent research has been exploring different types of chromatic numbers of the Cartesian product graphs such as list chromatic number \cite{BOROWIECKI20061955}, packing chromatic number \cite{BRESAR20072303} and $b$-chromatic number \cite{Balakrishnan2014,GUO201882}.

We first recall some basic notations and definitions needed in this article. 
Let $G$ be a graph.
For a subset $U$ of $V(G)$, $G[U]$ denotes the subgraph induced by $U$. 
A vertex coloring $c$ of $G$ is a \textit{proper coloring} if each pair of adjacent vertices has distinct colors.
The \textit{chromatic number} of $G$, denoted by $\chi(G)$, is the minimum number of colors needed so that $(G,c)$ is properly colored.
For each vertex $u\in V(G)$, we use notation $d_G(u)$ for the degree of $u$ in $G$.
Throughout this article, we let $P_n$ be a path with $n$ vertices, $K_n$ be a complete graph with $n$ vertices and $C_n$ be a cycle with $n$ vertices.
We let $S_{1,n}$ be a star with $n$ pendants.
For graphs $G$ and $H$, the \textit{Cartesian product} of $G$ and $H$, denoted by $G\square H$, is a graph where $V(G\square H)=V(G)\times V(H)$ and $uv\in E(G\square H)$ if either $x=x'$ and $yy'\in E(H)$ or $y=y'$ and $xx'\in E(G)$ for $u=(x,y)$ and $v=(x',y')$.

In this work, we give a structure of the $\delta$-complement of the finite Cartesian products of graphs.
Sharp bounds on the $\delta$-chromatic number (the chromatic number of $\delta$-complement) of the finite Cartesian products of graphs are also given.
In addition, we determine the specific value of the $\delta$-chromatic numbers of various classes of the Cartesian product of well-known graphs such as cycle, path, and star.

\section{Preliminary results}
In this section, we review some basic definitions and previous results.

\begin{definition}[\cite{math10081203}] \normalfont 
The \textit{$\delta$-complement} of a graph $G$, denoted $G_\delta$, is a graph obtained from $G$ by using the same vertex set and the following edge conditions: $uv\in E(G_\delta)$ if
\begin{enumerate}
    \item  $d(u)=d(v)$ in $G$ and $uv\notin E(G)$, or
    \item $d(u)\neq d(v)$ in $G$ and $uv\in E(G)$.
\end{enumerate}
\end{definition}

\begin{definition}[\cite{Vichitkunakorn2023}] A $\delta$-\textit{chromatic number} $\chi_\delta(G)$ of a graph $G$ is the chromatic number of $G_{\delta}$.
\end{definition}

Results on the $\delta$-chromatic numbers of some important graphs are
$\chi_\delta(P_n) = \ceil*{\frac{n-2}{2}}$ for $n\geq 5$ \cite{Vichitkunakorn2023},  $\chi_\delta(C_n) = \ceil*{\frac{n}{2}}$ \cite{Vichitkunakorn2023}, and
$\chi_\delta(W_n) = 1 + \chi_\delta(C_n) = 1 + \ceil*{\frac{n}{2}}$.

\begin{theorem}[\cite{sabidussi_1957}]
\label{thm: chiProduct}
    Let $G$ and $H$ be graphs. We have $\chi(G\square H)=\max\{\chi(G),\chi(H)\}$.
\end{theorem}

\begin{theorem}[\cite{Vichitkunakorn2023}]
\label{thm: NG_bound}
    For $n\geq 4$, let $G$ be a graph with $n$ vertices. Let $d_1,\dots,d_m$ be all distinct values of the degrees of the vertices in $G$. Partition $V(G)$ into non-empty sets $V_{d_1},V_{d_2},\dots, V_{d_m}$. We have 
    \[
    \max_{1\leq i\leq m}\{|V_{d_i}|\}\leq \chi(G)\cdot \chi_{\delta}(G)\leq \left(\frac{m+n}{2}\right)^2
    \]
    and 
    \[
    2\cdot \sqrt{\max_{1\leq i\leq m} \{|V_{d_i}|\}}\leq \chi(G)+\chi_{\delta}(G)
\leq m+n.
    \]
\end{theorem}

\section{Structure of the $\delta$-complements of Cartesian products}
This section contains the structure of the $\delta$-complement of the Cartesian product of graphs.

The following theorem shows that the edge set of the $\delta$-complements of the Cartesian product contains the edge set of the Cartesian product of the $\delta$-complements of graphs.
It is a fundamental result that will be used throughout what follows.

\begin{theorem}\label{thm: cartesianEdge}
    For graphs $G$ and $H$, we have $(G\square H)_\delta=(V,E)$ where $V=V(G\square H)$ and $E=E(G_{\delta}\square H_{\delta})\cup S$ where $S=\{uv: u=(u_1,u_2)\in V(G\square H) \text{ and } v=(v_1,v_2)\in V(G\square H) \text{ where } u_1\neq v_1, u_2\neq v_2 \text{ and } d_{G\square H}(u)=d_{G\square H}(v)\}$.
\end{theorem}
\begin{proof}
$(\Longrightarrow)$ 
Let $u=(u_1,u_2)$ and $v=(v_1,v_2)$ be distinct vertices in $(G\square H)_\delta$ where $uv\in E((G\square H)_\delta)$.
    It follows that either 
    \begin{itemize}
        \item $uv\in E(G\square H)$ and $d_{G\square H}(u)\neq d_{G\square H}(v)$, or
        \item $uv\not\in E(G\square H)$ and $d_{G\square H}(u)=d_{G\square H}(v)$.
    \end{itemize}
    
In case $uv\in E(G\square H)$ and $d_{G\square H}(u)\neq d_{G\square H}(v)$, without loss of generality, we suppose that $u_1=v_1$, $u_2\neq v_2$ and $u_2v_2\in E(H)$.
Since $d_G(u_1)=d_G(v_1)$, it follows that $d_H(u_2)\neq d_H(v_2)$.
Thus $u_2v_2\in E(H_{\delta})$.
Hence $uv\in E(G_{\delta}\square H_{\delta})$.
In case $uv\not\in E(G\square H)$ and $d_{G\square H}(u)=d_{G\square H}(v)$, we have $u_1\neq v_1$ and $u_2\neq v_2$.
Thus $uv\in S$.

$(\Longleftarrow)$ Let $uv\in E(G_{\delta}\square H_{\delta})\cup S$.
Consider $uv\in E(G_{\delta}\square H_{\delta})$.
Without loss of generality, we suppose that $u_1=v_1$, $u_2\neq v_2$ and $u_2v_2\in E(H_{\delta})$.
If $d_{G\square H}(u)=d_{G\square H}(v)$, then $d_H(u_2)=d_H(v_2)$.
Thus $u_2v_2\not\in E(H)$.
Hence $uv\not\in E(G\square H)$.
Since $d_{G\square H}(u)=d_{G\square H}(v)$, we have $uv\in E((G\square H)_\delta)$.
If $d_{G\square H}(u)\neq d_{G\square H}(v)$, then $d_H(u_2)\neq d_H(v_2)$.
Thus $u_2v_2\in E(H)$ and $uv\in E(G\square H)$.
Since $d_{G\square H}(u)\neq d_{G\square H}(v)$, we have $uv\in E((G\square H)_\delta)$.
Now, we consider $uv\in S$.
We have that $u_1\neq v_1$ and $u_2\neq v_2$. So $uv\notin E(G\square H)$.
Since $d_{G\square H}(u)=d_{G\square H}(v)$, it follows that $uv\in E((G\square H)_\delta)$.
\end{proof}

\begin{corollary}
    Let $G$ and $H$ be graphs.
    We have $(G\square H)_\delta=G_{\delta}\square H_{\delta}$ if and only if
    for any $u=(u_1,u_2)$ and $v=(v_1,v_2)$ in $V(G\square H)$ where $u_1\neq v_1$ and $u_2\neq v_2$, we have $d_{G\square H}(u) \neq d_{G\square H}(v)$.
\end{corollary}


In general, we have the following theorem for a finite Cartesian product of graphs.
\begin{theorem}\label{thm: cartesianEdge_super}
    For graphs $G_1,\dots,G_k$, we have $(G_1\square\cdots\square G_k)_\delta=(V,E)$ where $V=V(G_1\square\cdots\square G_k)$ and $E=E((G_1)_{\delta}\square\cdots\square (G_k)_{\delta})\cup S$ 
    such that $S$ is the set of $uv$ where $u=(u_1,\dots,u_k)\in V$, $v=(v_1,\dots,v_k)\in V$, there are at least two indices $i$ that $u_i\neq v_i$, and $d_{G_1\square\cdots\square G_k}(u)=d_{G_1\square\cdots\square G_k}(v)$.
\end{theorem}
\begin{proof}
    It is well-known that two vertices $u=(u_1,\dots,u_k)$ and $(v_1,\dots,v_k)$ in $G_1\square\dots\square G_k$ are adjacent if and only if there is exactly one $i$ such that $u_i \neq v_i$ and $u_iv_i \in E(G_i)$.
    The rest of the proof follows similar arguments as in Theorem~\ref{thm: cartesianEdge}.
\end{proof}

The following three results are applications of Theorem~\ref{thm: cartesianEdge_super}.
\begin{theorem}
     $(G_1\square\cdots\square G_k)_\delta=(G_1)_{\delta}\square\cdots\square (G_k)_{\delta}$ if and only if there are at most one $i$ such that $G_i \neq K_1$. 
\end{theorem}
\begin{proof}
    From Theorem~\ref{thm: cartesianEdge_super}, we need to show that $S=\emptyset$ if and only if there are at most one $i$ such that $G_i \neq K_1$.
    
    Suppose that there are $i\neq j$ such that $G_i\neq K_1$ and $G_j\neq K_1$. 
    Choose $u=(u_1,\dots,u_k)$ and $v=(v_1,\dots,v_2)$ such that 
    $u_i\neq v_i$, $u_j \neq v_j$, $d_{G_i}(u_i)=d_{G_i}(v_i)$,  $d_{G_j}(u_j)=d_{G_j}(v_j)$ and $u_\ell = v_\ell$ for all $\ell \not\in \{i,j\}$.
    So $d_{G_1\square\cdots\square G_k}(u)=d_{G_1\square\cdots\square G_k}(v)$.
    Then $uv \in S$. Hence $S\neq \emptyset$.

    The converse is obvious.
\end{proof}

\begin{corollary}
     $(G\square H)_\delta=G_{\delta}\square H_{\delta}$ if and only if  $G=K_1$ or $H=K_1$.
\end{corollary}

\section{Bounds on the $\delta$-chromatic numbers of Cartesian products}

In this section, we provide some exact numbers and bounds on the $\delta$-chromatic numbers of some common graphs.

\begin{theorem}
Let $G_1,\dots,G_k$ be graphs. We have
    $$\max\{\chi_{\delta}(G_1),\dots,\chi_{\delta}(G_k)\}\leq \chi_{\delta}(G_1\square\cdots\square G_k).$$
\end{theorem}
\begin{proof}
   The proof follows directly from Theorem \ref{thm: chiProduct} and \ref{thm: cartesianEdge_super}.
\end{proof}

\begin{theorem}\label{thm: DistinctDegDif}
Let $G$ and $H$ be graphs. If any positive degree difference of vertices in $G$ is not equal to that of in $H$, then
    \[
    \chi_\delta(G\square H) \leq n_{\max}(H)\cdot \max( \chi_\delta(G) , m(H) )
    \]
    where $n_{\max}(H)$ denotes the maximum number of vertices of the same degree in $H$ and $m(H)$ is the number of different degrees in $H$. 
    Furthermore, the bound is sharp.
\end{theorem}
\begin{proof}
    By Theorem~\ref{thm: cartesianEdge} and the assumption that any positive degree difference of vertices in $G$ is not equal to that of in $H$, the edges in $S$ are $uv$ where $u=(u_1,u_2)$, $v=(v_1,v_2)$ such that $u_1\neq v_1$, $u_2\neq v_2$, $d_G(u_1)=d_G(v_1)$ and $d_H(u_2)=d_H(v_2)$.
    We partition $V(H)$ according to vertex degree into $W_1, W_2, \dots, W_{m(H)}$.
    Write $W_j = \{h_{j,1}, h_{j,2}, \dots, h_{j,n_j} \}$  for $1\leq j\leq m(H)$.

    Define $p=\max( \chi_\delta(G) , m(H) )$.
    Let $c_0:V(G)\to \{1,2,\dots,\chi_\delta(G)\}$ be a proper coloring of $G_\delta$.
    We define a coloring $c : V(G)\times V(H) \to \{1,2,\dots, n_{\max}(H)\cdot p\}$ as 
    \[ c(g, h_{j,k}) = f(g,j) + (k-1)p, \]
    for $k=1,\dots,n_j$, 
    where $f(g,j)\in\{1,2,\dots,p\}$ and $f(g,j) \equiv c_0(g) + j - 1 \pmod p$.
    The first copy of $G$ in $W_1$ gets the original coloring $c_0$, while we keep adding $p$ to the coloring of each other copy of $G$ in $W_1$.
    In other $W_j$, we perform different cyclic permutations modulo $p$ to $c_0$ and assign it to the first copy of $G$ in $W_j$.
    See Table~\ref{tab: DistinctDegDif} for an example.
    We see that the vertices in the same copy of $G$ received a coloring equivalent to $c_0$ and a cyclic permutation modulo $p$ up to an additive constant $(k-1)p$ for some $k=1,\dots,n_j$.
    For a fixed $g\in V(G)$, the vertices in the same copy of $H$, written in the form $(g,h_{j,k})$ where $1\leq j\leq m(H)$ and $1\leq k \leq n_j $, received distinct colors because $j \leq p$ and $k\leq n_{\max}(H)$.
    
    Lastly, any endpoints of an edge in $S$ are of the form $(g, h_{j,k})$ and $(g', h_{j,k'})$ where $g\neq g'$ and $k\neq k'$, which received different colors as $k\neq k'$. 
    The sharpness of the bound appears in Theorem \ref{thm: CP}.
\end{proof}

\begin{table}
    \centering
    \begin{tikzpicture}
      \matrix(table)[
      matrix of nodes,
      row sep =-\pgflinewidth,
      column sep = -\pgflinewidth,
      nodes={anchor=center,
             minimum width=20pt, 
             minimum height=20pt, 
             },
      ] 
      {
        & $h_{1,1}$ & $h_{1,2}$ & $h_{2,1}$ & $h_{3,1}$ & $h_{3,2}$ & $h_{4,1}$ & $h_{4,2}$\\
        $g_1$ & 1 & 5 & 2 & 3 & 7 & 4 & 8 \\
        $g_2$ & 3 & 7 & 4 & 1 & 5 & 2 & 6 \\
        $g_3$ & 1 & 5 & 2 & 3 & 7 & 4 & 8 \\
        $g_4$ & 2 & 6 & 3 & 4 & 8 & 1 & 5 \\
        $g_5$ & 3 & 7 & 4 & 1 & 5 & 2 & 6 \\
      };
      \draw (table-2-2.north west) -- (table-2-8.north east) -- (table-6-8.south east) -- (table-6-2.south west) -- cycle ;
      \draw (table-2-4.north west) -- (table-6-4.south west);
      \draw (table-2-5.north west) -- (table-6-5.south west);
      \draw (table-2-7.north west) -- (table-6-7.south west);
    \end{tikzpicture}
    \medskip
    \caption{An example of a coloring in the proof of Theorem~\ref{thm: DistinctDegDif} where $\chi_\delta(G)=3$, $m(H)=4$ and $n_{\max}(H)=2$.}
    \label{tab: DistinctDegDif}
\end{table}

\begin{corollary} \label{cor: GP_3}
    Let $G$ be a graph with $\chi_\delta(G) \geq 2$.
    If $d_G(v)\neq d_G(u)+1$ for all $u,v\in V(G)$, then $\chi_{\delta}(G\square P_3)\leq 2\chi_{\delta}(G)$.
\end{corollary}



\begin{theorem}\label{thm: CP}
    For $n\geq 5$, we have $\chi_{\delta}(C_n\square P_3)=2\chi_{\delta}(C_n)=2\left\lceil\frac{n}{2}\right\rceil$.
\end{theorem}
\begin{proof}
Let $P_3=v_1v_2v_3$.
It is easy to see that $\chi_{\delta}(C_n)=\chi(\overline{C_n}) = \left\lceil\frac{n}{2}\right\rceil$.
Since $C_n$ is regular, it also follows that 
$d_{C_n\square P_3}(u,v_1) = d_{C_n\square P_3}(w,v_3)$
for all $u,w\in V(C_n)$.
Since $(u,v_1)$ is not adjacent to $(w,v_3)$ in $C_n\square P_3$, it follows that each vertex in the first copy of $C_n$ is adjacent to all the vertices in the third copy of $C_n$ in $(C_n\square P_3)_\delta$.
Hence the colors used in the two copies do not coincide.
Thus $\chi_\delta(C_n\square P_3)\geq 2\chi_{\delta}(C_n)$.
By Corollary~\ref{cor: GP_3}, we can conclude that $\chi_\delta(C_n\square P_3)= 2\chi_{\delta}(C_n)$.
\end{proof}



\begin{example} \label{ex: poor_star_P3}
    For $m\geq 3$, we have $\chi_\delta(S_{1,m}\square P_3) \leq 2(m+1)$.
\end{example}
\begin{proof}
    Let $G=S_{1,m}$ and $H=P_3$. Theorem~\ref{thm: DistinctDegDif} gives the desired upper bound.
\end{proof}

The bound in Example~\ref{ex: poor_star_P3} is not sharp.
When $G = K_1 \vee H$ is a join of a singleton and a regular graph $H$,
the following theorem gives an improved upper bound on $\chi_{\delta}(G\square P_3)$ in terms of $\chi_\delta(H)$.
Examples of the graph $G$ include stars $S_{1,m} = K_1\vee N_m$ (in Theorem~\ref{thm: StarPath}), wheels $W_{m}=K_1\vee C_m$, and windmills $K_1\vee mK_n$.

\begin{theorem} \label{thm:uvH}
    Let $H$ be a $k$-regular graph.
    Let $G=\{u\} \vee H$ be the join of a singleton and $H$.
    Suppose $|V(H)|\geq 3$ and $\chi_{\delta}(H)\geq 2$.
    If $|V(H)| > k+2$, then $\chi_{\delta}(G\square P_3)\leq 2\chi_{\delta}(H)$.
\end{theorem}
\begin{proof}
    Let $r\in V(H)$. 
    Note that $d_G(u)=|V(H)|$.
    Since $d_{G\square P_3}(r,v_i)=d_G(r)+d_{P_3}(v_i)
    \leq k+3
    < d_G(u)+1\leq d_{G\square P_3}(u,v_j)$ for $i,j\in \{1,2,3\}$, we have $(r,v_i)$ and $(u,v_j)$  are adjacent in $(G\square P_3)_\delta$ if and only if $i=j$.

    Let $H_\delta^i$ be the $i$-th copy of $H_\delta$ in $G_\delta$ for $i=1,2,3$.
    Next, we construct a proper coloring $c$ as follows.
    We trivially color $H_{\delta}^1$, as a copy of $H_\delta$, by a $\chi_{\delta}(H)$-coloring.
    Since 
    each vertex in $H_{\delta}^1$ is adjacent to any vertices in $H_{\delta}^3$,
    it requires $2\chi_{\delta}(H)$ colors for $H^1_{\delta}$ and $H^3_{\delta}$.
    We color $H^3_{\delta}$ using $c(r,v_3)=c(r,v_1)+\chi_{\delta}(H)$.
    For $i=1,3$, we notice that a vertex $(r,v_i)\in V(H^i_{\delta})$ and $(s,v_2)\in V(H^2_{\delta})$ are adjacent if and only if $r= s$. 
    We let $c(r,v_2)=c(r,v_1)+1$ if $1\leq c(r,v_1)\leq \chi_{\delta}(H)-1$; otherwise, $c(r,v_2)=1$.
    Lastly, we color $(u,v_1)$, $(u,v_2)$ and $(u,v_3)$ by $\chi_\delta(H)+1$, $\chi_\delta(H)+2$ and $1$, respectively.  
    This gives a proper coloring of $(G\square P_3)_\delta$ with $2\chi_{\delta}(H)$ colors.
\end{proof}

The sharpness of the bound in Theorem~\ref{thm:uvH} will be shown in Theorem~\ref{thm: StarPath}.

\section{The $\delta$-chromatic numbers of the Cartesian products of some graphs}

In this section, we give the exact values of the $\delta$-chromatic numbers of the Cartesian products of stars and paths.


\begin{theorem}\label{thm: StarStar}
    $\chi_\delta(S_{1,m}\square S_{1,n}) = mn$ for $m,n\geq 3$.
\end{theorem}
\begin{proof}
    Let $V(S_{1,n})=\{0,1,\dots,k\}$ where $d_{S_{1,k}}(0)=k$ for $k=n,m$.
    An $mn$-coloring on $(S_{1,m}\square S_{1,n})_\delta$ is
    \[ c(i,j) =
    \begin{cases}
        j+1         & \text{if } i=0, \\
        (i-1)n + j & \text{if } 1\leq i \leq m \text{ and } 1\leq j \leq n,\\
        (i+1)n     & \text{if } 1\leq i < m \text{ and } j=0, \\
        n+2     & \text{if } i=m \text{ and } j=0,
    \end{cases}\]
    as shown in Fig. \ref{fig:S1m-S1n}.
    In addition, the set $\{(i,j) : 1\leq i \leq m, 1\leq j\leq n\}$ forms an $mn$-clique in $(S_{1,m}\square S_{1,n})_\delta$.
\end{proof}

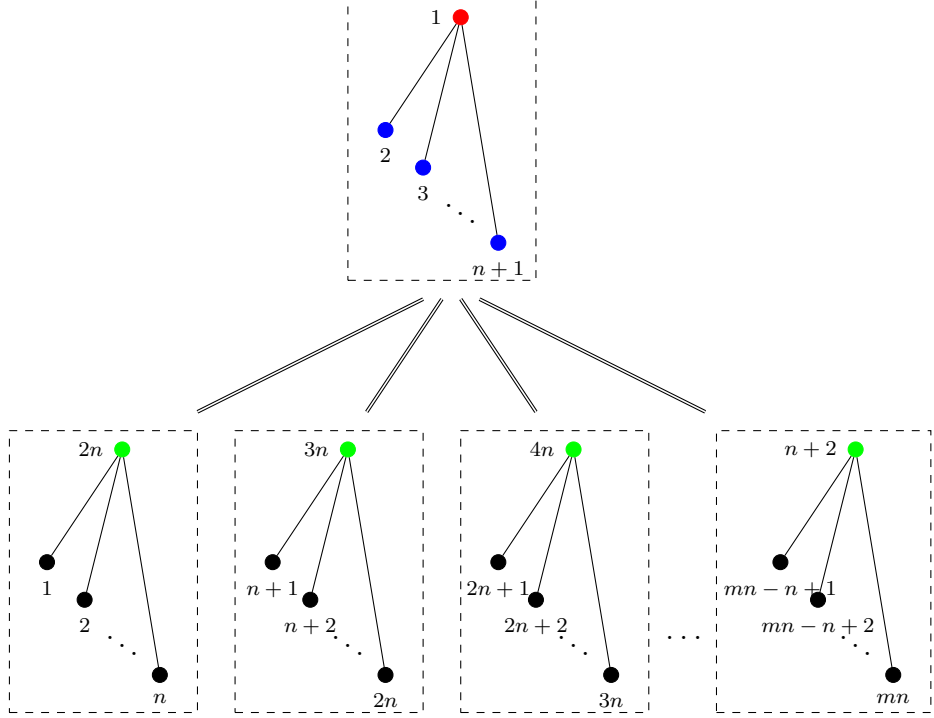
\begin{figure}
    \centering
    \begin{tikzpicture}[
        scale=0.5,
        every label/.append style={font=\tiny}
    ]
    \tikzstyle{pt}=[draw,shape=circle,inner sep=2pt,fill=black];
    \begin{scope}[shift={(6,11.5)}]
        \node[pt,red] (v0) at (2,-1) [label=left:$1$] {};
        \node[pt,blue] (v1) at (0,-4) [label=below:$2$] {};
        \node[pt,blue] (v2) at (1,-5) [label=below:$3$] {};
        \node at (2,-6) {$\ddots$};
        \node[pt,blue] (vm) at (3,-7) [label=below:$n+1$] {};
        \draw (v0) -- (v1) (v0) -- (v2) (v0) -- (vm);
        \draw[dashed] (-1,-8) rectangle (4,-.5);
    \end{scope}
    \begin{scope}[shift={(-3,0)}]
        \node[pt,green] (v0) at (2,-1) [label=left:$2n$] {};
        \node[pt,black] (v1) at (0,-4) [label=below:$1$] {};
        \node[pt,black] (v2) at (1,-5) [label=below:$2$] {};
        \node at (2,-6) {$\ddots$};
        \node[pt,black] (vm) at (3,-7) [label=below:$n$] {};
        \draw (v0) -- (v1) (v0) -- (v2) (v0) -- (vm);
        \draw[dashed] (-1,-8) rectangle (4,-.5);
    \end{scope}
    \begin{scope}[shift={(3,0)}]
        \node[pt,green] (v0) at (2,-1) [label=left:$3n$] {};
        \node[pt,black] (v1) at (0,-4) [label=below:$n+1$] {};
        \node[pt,black] (v2) at (1,-5) [label=below:$n+2$] {};
        \node at (2,-6) {$\ddots$};
        \node[pt,black] (vm) at (3,-7) [label=below:$2n$] {};
        \draw (v0) -- (v1) (v0) -- (v2) (v0) -- (vm);
        \draw[dashed] (-1,-8) rectangle (4,-.5);
    \end{scope}
    \begin{scope}[shift={(9,0)}]
        \node[pt,green] (v0) at (2,-1) [label=left:$4n$] {};
        \node[pt,black] (v1) at (0,-4) [label=below:$2n+1$] {};
        \node[pt,black] (v2) at (1,-5) [label=below:$2n+2$] {};
        \node at (2,-6) {$\ddots$};
        \node[pt,black] (vm) at (3,-7) [label=below:$3n$] {};
        \draw (v0) -- (v1) (v0) -- (v2) (v0) -- (vm);
        \draw[dashed] (-1,-8) rectangle (4,-.5);
    \end{scope}
    \node at (14,-6) {$\dots$};
    \begin{scope}[shift={(16.5,0)}]
        \node[pt,green] (v0) at (2,-1) [label=left:$n+2$] {};
        \node[pt,black] (v1) at (0,-4) [label=below:$mn-n+1$] {};
        \node[pt,black] (v2) at (1,-5) [label=below:$mn-n+2$] {};
        \node at (2,-6) {$\ddots$};
        \node[pt,black] (vm) at (3,-7) [label=below:$mn$] {};
        \draw (v0) -- (v1) (v0) -- (v2) (v0) -- (vm);
        \draw[dashed] (-1.7,-8) rectangle (4,-.5);
    \end{scope}
    \draw[double] (7,3) -- +(-6,-3) ;
    \draw[double] (7.5,3) -- +(-2,-3) ;
    \draw[double] (8,3) -- +(2,-3) ;
    \draw[double] (8.5,3) -- +(6,-3) ;
    \end{tikzpicture}
    \caption{A proper $mn$-coloring of $(S_{1,m}\square S_{1,n})_\delta$. The vertices of the same degree in $S_{1,m}\square S_{1,n}$ are indicated by the same color  and are pairwise adjacent in $(S_{1,m}\square S_{1,n})_\delta$. 
    Each double line denotes the edges connecting 
    the corresponding vertices between two copies of $S_{1,n}$.
    Note that the blue and the green will have the same degree when $m=n$.}
    \label{fig:S1m-S1n}
\end{figure}

\begin{theorem} \label{thm: StarPath}
    $\chi_\delta(S_{1,m}\square P_n) = m \left\lceil \frac{n-2}{2} \right\rceil$ for $m\geq 3$ and $n\geq 3$.
\end{theorem}
\begin{proof}
    Let $V(S_{1,m})=\{0,1,\dots,m\}$ where $d_{S_{1,m}}(0)=m$ and
    Let $V(P_n)=\{1,2,\dots,n\}$ where $d_{P_n}(1)=d_{P_n}(n)=1$. 
    
    When $n=3$, Theorem~\ref{thm:uvH} gives $\chi_\delta(S_{1,m}\square P_3) \leq 2m$.
    Since the set $\{(i,j) \in V((S_{1,m}\square P_3)_\delta) : 1\leq i\leq m, j=1,3\}$ forms a $2m$-clique in $(S_{1,m}\square P_3)_\delta$, we get $\chi_\delta(S_{1,m}\square P_3) = 2m$.

    When $n=4$, Theorem~\ref{thm: DistinctDegDif} with $G=P_4$ and $H=S_{1,m}$ gives $\chi_\delta(S_{1,m}\square P_4) = \chi_\delta(P_4 \square S_{1,m}) \leq 2m$.
    The set $\{(i,j)\in V((S_{1,m}\square P_4)_\delta) : 1\leq i\leq m, j=1,4\}$ forms a $2m$-clique in $(S_{1,m}\square P_4)_\delta$.
    Hence $\chi_\delta(S_{1,m}\square P_4) = 2m$.

    When $n\geq 5$, we let $k=\ceil{\frac{n-2}{2}}$. A coloring is 
    \[c(i,j) = \begin{cases}
        i + (k-1)m  & 0\leq i \leq m \text{ and } j=1, \\
        i  & 1\leq i \leq m \text{ and } j=n,\\
        i + (\floor{j/2}-1)m  & 0\leq i \leq m \text{ and }2\le j\le n-1 \text{ where }(i,j)\neq(0,2),(0,3), \\
        km & i=0 \text{ and } j=2,3,n\\
        
    \end{cases}\]
    as shown in Fig.~\ref{fig:S1m-Pn}.
    We thus have a proper $km$-coloring of $(S_{1,m}\square P_n)_\delta$. 
    In addition, the set $\{ (i,j)\in V((S_{1,m}\square P_n)_\delta) : 1\leq i\leq m\text{ and } 2\leq j\leq n-1 \text{ and } j \text{ is even} \}$ forms a clique of size $m\left\lceil \frac{n-2}{2} \right\rceil$ in $(S_{1,m}\square P_n)_\delta$.
\end{proof}

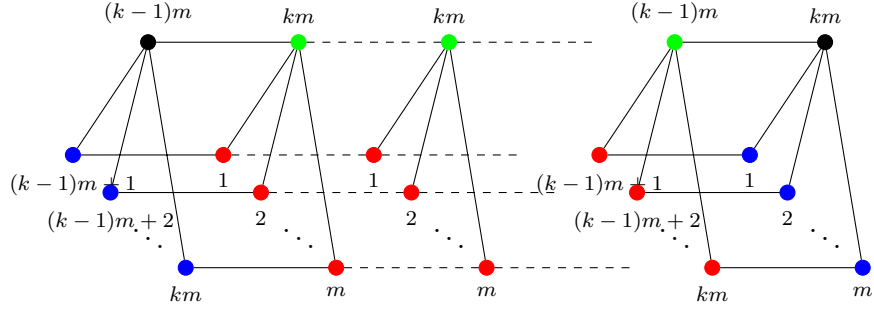
\begin{figure}
    \centering

    \begin{tikzpicture}[
        scale=0.5,
        every label/.append style={font=\tiny}
    ]
    \tikzstyle{pt}=[draw,shape=circle,inner sep=2pt,fill=white];
    \begin{scope}
        \node[pt,black] (v0) at (2,-1) [label=above:$(k-1)m$] {};
        \node[pt,blue] (v1) at (0,-4) [label=below:$(k-1)m+1$] {};
        \node[pt,blue] (v2) at (1,-5) [label=below:$(k-1)m+2$] {};
        \node at (2,-6) {$\ddots$};
        \node[pt,blue] (vm) at (3,-7) [label=below:$km$] {};
        \draw (v0) -- (v1) (v0) -- (v2) (v0) -- (vm) (v0) -- ++(4,0) (v1) -- ++(4,0) (v2) -- ++(4,0) (vm) -- ++(4,0);
    \end{scope}
    \begin{scope}[shift={(4,0)}]
        \node[pt,green] (v0) at (2,-1) [label=above:$km$] {};
        \node[pt,red] (v1) at (0,-4) [label=below:$1$] {};
        \node[pt,red] (v2) at (1,-5) [label=below:$2$] {};
        \node at (2,-6) {$\ddots$};
        \node[pt,red] (vm) at (3,-7) [label=below:$m$] {};
        \draw (v0) -- (v1) (v0) -- (v2) (v0) -- (vm);
        \draw[dashed] (v0) -- ++(4,0) (v1) -- ++(4,0) (v2) -- ++(4,0) (vm) -- ++(4,0);
    \end{scope}
    \begin{scope}[shift={(8,0)}]
        \node[pt,green] (v0) at (2,-1) [label=above:$km$] {};
        \node[pt,red] (v1) at (0,-4) [label=below:$1$] {};
        \node[pt,red] (v2) at (1,-5) [label=below:$2$] {};
        \node at (2,-6) {$\ddots$};
        \node[pt,red] (vm) at (3,-7) [label=below:$m$] {};
        \draw (v0) -- (v1) (v0) -- (v2) (v0) -- (vm);
        \draw[dashed] (v0) -- ++(4,0) (v1) -- ++(4,0) (v2) -- ++(4,0) (vm) -- ++(4,0);
    \end{scope}
    \begin{scope}[shift={(14,0)}]
        \node[pt,green] (v0) at (2,-1) [label=above:$(k-1)m$] {};
        \node[pt,red] (v1) at (0,-4) [label=below:$(k-1)m+1$] {};
        \node[pt,red] (v2) at (1,-5) [label=below:$(k-1)m+2$] {};
        \node at (2,-6) {$\ddots$};
        \node[pt,red] (vm) at (3,-7) [label=below:$km$] {};
        \draw (v0) -- (v1) (v0) -- (v2) (v0) -- (vm) (v0) -- ++(4,0) (v1) -- ++(4,0) (v2) -- ++(4,0) (vm) -- ++(4,0);
    \end{scope}
    \begin{scope}[shift={(18,0)}]
        \node[pt,black] (v0) at (2,-1) [label=above:$km$] {};
        \node[pt,blue] (v1) at (0,-4) [label=below:$1$] {};
        \node[pt,blue] (v2) at (1,-5) [label=below:$2$] {};
        \node at (2,-6) {$\ddots$};
        \node[pt,blue] (vm) at (3,-7) [label=below:$m$] {};
        \draw (v0) -- (v1) (v0) -- (v2) (v0) -- (vm);
    \end{scope}
    \end{tikzpicture}
    \caption{A proper $km$-coloring of $(S_{1,m}\square P_n)_\delta$ where $k = \ceil{\frac{n-2}{2}}$. The vertices of the same degree in $S_{1,m}\square P_n$ are indicated by the same color and are pairwise adjacent except for the pairs with a dashed line.}
    \label{fig:S1m-Pn}
\end{figure}

The following lemma is crucial for proving Theorem~\ref{thm: pathpath}.
\begin{lemma}\label{lem: bound}
For $n\geq 6$ and $k\geq 8$, we have 
\[
2\left\lceil\frac{n-2}{2}\right\rceil+2\left\lceil\frac{k-2}{2}\right\rceil+1<\left\lceil\frac{(n-2)(k-2)}{2}\right\rceil.
\]
\end{lemma}
\begin{proof}
Suppose $2\left\lceil\frac{n-2}{2}\right\rceil+2\left\lceil\frac{k-2}{2}\right\rceil+1\geq \left\lceil\frac{(n-2)(k-2)}{2}\right\rceil$.

{\case $n$ and $k$ are even.}\\
We have
\begin{align*}
    n+k-3     &\geq \frac{(n-2)(k-2)}{2},\\
    2n+2k-6 &\geq nk-2n-2k+4.
\end{align*}
Thus $n\leq \frac{4k-10}{k-4}<6$ when $k\geq 8$, which is a contradiction.

{\case $n$ is odd and $k$ is even.}\\
We have
\begin{align*}
    n+k-2     &\geq \frac{(n-1)(k-2)}{2},\\
    2n+2k-4 &\geq nk-2n-k+2.
\end{align*}
Thus $n\leq \frac{3k-6}{k-4}\leq \frac{9}{2}$, which is not possible when $k\geq 8$.
The same argument can be applied when $n$ is even and $k$ is odd.

{\case $n$ and $k$ are odd.}\\
\begin{align*}
    n+k-1     &\geq \frac{(n-1)(k-1)}{2},\\
    2n+2k-4 &\geq nk-n-k+1.
\end{align*}
Thus $n\leq \frac{3k-5}{k-3} \leq \frac{19}{5}$, which is not possible when $k\geq 8$.

Therefore $2\left\lceil\frac{n-2}{2}\right\rceil+2\left\lceil\frac{k-2}{2}\right\rceil+1<\left\lceil\frac{(n-2)(k-2)}{2}\right\rceil.$
\end{proof}

\begin{theorem}
\label{thm: pathpath}
    For $6\leq n\leq k$, we have 
    \[
\chi_{\delta}(P_{n}\square P_{k})=\left\lceil\frac{(n-2)(k-2)}{2}\right\rceil.
\]
\end{theorem}
\begin{proof}
    Let $V_{d}$ be the set of vertices of degree $d$ in $P_n\square P_k$.
    The vertex set of $P_n\square P_k$ can be partitioned into $V_2, V_3$ and $V_4$.
    We note that $V(P_n\square P_k)=V((P_n\square P_k)_\delta)=V(P_n)\times V(P_k)$.
    Let $(i,j)\in V(P_n\square P_k)$ for $i=1,\dots,n$ and $j=1,\dots,k$.
    We have that $V_3=\{(i,j): i=1,n \text{ and } 2\leq j\leq k-1\}\cup\{(i,j):j=1,k \text{ and } 2\leq i\leq n-1\}$ and $V_4=\{(i,j): 2\leq i\leq n-1 \text{ and } 2\leq j\leq k-1\}$.
    Thus $|V_2|=4$, $|V_3|=2(n+k-4)$ and $|V_4|=(n-2)(k-2)$.
   The vertices $(i,j)$ and $(i',j')$ are adjacent in $P_n\square P_k$ if and only if $|i-i'|+|j-j'|=1$.
    Thus 
    \begin{itemize}
 %
        \item if $d_{P_n\square P_k}(i,j)=d_{P_n\square P_k}(i',j')$, then the vertices $(i,j)$ and $(i',j')$ are adjacent in $(P_n\square P_k)_\delta$ if and only if $|i-i'|+|j-j'|\geq 2$,
        \item if $d_{P_n\square P_k}(i,j)\neq d_{P_n\square P_k}(i',j')$, then the vertices $(i,j)$ and $(i',j')$ are adjacent in $(P_n\square P_k)_\delta$ if and only if $|i-i'|+|j-j'|=1$.
    \end{itemize}
Since each pair of vertices $(i,j),(i',j')\in V_4$ with $|i-i'|+|j-j'|\geq 2$ are adjacent, it follows that
\[
\chi_{\delta}(P_n\square P_k)\geq \omega((P_n\square P_k)_\delta)\geq \left\lceil\frac{(n-2)(k-2)}{2}\right\rceil.
\]
We note that
\[
\left\lceil\frac{(n-2)(k-2)}{2}\right\rceil=
\begin{cases} 
\frac{(n-2)(k-2)}{2} &\text{ if } n \text{ or } k \text{ is even},\\
\left\lfloor\frac{k-2}{2}\right\rfloor(n-2)+\left\lceil\frac{n-2}{2}\right\rceil &\text{ if } n \text{ and } k \text{ are odd}.
\end{cases}
\]
    We color $V_4$ by a coloring $c_0$ defined by
    \[
    c_0(i,j)=
    \begin{cases}
        (i-2)\left\lfloor\frac{k-2}{2}\right\rfloor+\left\lfloor\frac{j-2}{2}\right\rfloor &\text{ if } i=2,\dots,n-1 \text{ and } j=2,\dots,2\left\lfloor\frac{k-2}{2}\right\rfloor+1,\\
    
        (n-2)\left\lfloor\frac{k-2}{2}\right\rfloor+\left\lfloor\frac{i-2}{2}\right\rfloor &\text{ if } k \text{ is odd and } j=k-1,i=1,\dots,n-2.
    \end{cases}
    \]
The coloring $c_0$ uses $\left\lceil\frac{(n-2)(k-2)}{2}\right\rceil$ colors.
Since the coloring in $V_4$ has at most 2 vertices with the same color and they are adjacent in $P_n\square P_k$, which are not adjacent in $(P_n\square P_k)_\delta$, the coloring $c_0$ on $(P_n\square P_k)_\delta[V_4]$ is proper.
The case $6\leq k \leq 7$ can be verified.
Now, we suppose that $k\geq 8$. 
We color $V_3$ using $2\left\lceil\frac{n-2}{2}\right\rceil+2\left\lceil\frac{k-2}{2}\right\rceil$ colors.
We color the vertices $V_3$ in pair of consecutive vertices (except possibly the last one in a block) clockwise starting from location $(1,2)$ to $(2,1)$. 
Each color in $V_3$ needs to avoid at most $2\left\lceil\frac{n-2}{2}\right\rceil+2\left\lceil\frac{k-2}{2}\right\rceil -1$ colors of the other vertices in $V_3$ and two neighbors per each color in $V_4$, i.e., we have to avoid $2\left\lceil\frac{n-2}{2}\right\rceil+2\left\lceil\frac{k-2}{2}\right\rceil +1$ colors.
By Lemma \ref{lem: bound}, there is a remaining color in $V_4$ that is available to assign to the considered vertex.
Since each vertex in $V_2$ has degree 5 in $(P_n\square P_k)_\delta$ and $\left\lceil\frac{(n-2)(k-2)}{2}\right\rceil>5$, we can color $V_2$. 
This completes the proof.
\end{proof}

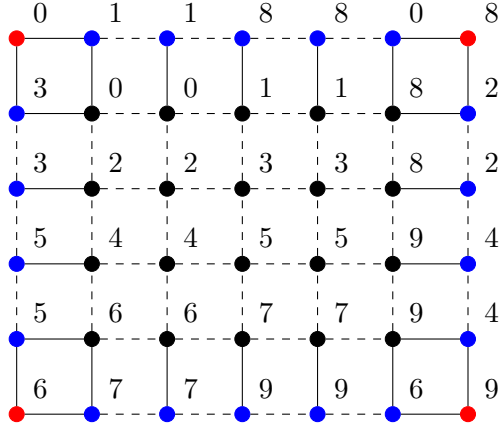
\begin{figure}
    \centering
    \begin{tikzpicture}[scale=1, every label/.append style={font=\small}]
    \tikzstyle{pt}=[draw,shape=circle,inner sep=2pt,fill=black];
    \begin{scope}[]
    	\node[pt,red] (1) at (0,0) [label=above right:6] {};
    	\node[pt,blue] (2) at (0,1) [label=above right:5] {};
    	\node[pt,blue] (3) at (0,2) [label=above right:5] {};
    	\node[pt,blue] (4) at (0,3) [label=above right:3] {};
    	\node[pt,blue] (5) at (0,4) [label=above right:3] {};
    	\node[pt,red] (6) at (0,5) [label=above right:0] {};
    	\draw (1) -- (2) (5) -- (6);
    	\draw[dashed] (2) -- (3) -- (4) -- (5);
    	\draw (1) -- +(1,0) (2) -- +(1,0) (3) -- +(1,0) (4) -- +(1,0) (5) -- +(1,0) (6) -- +(1,0);
    \end{scope}
    \begin{scope}[shift={(1,0)}]
    	\node[pt,blue] (1) at (0,0) [label=above right:7] {};
    	\node[pt] (2) at (0,1) [label=above right:6] {};
    	\node[pt] (3) at (0,2) [label=above right:4] {};
    	\node[pt] (4) at (0,3) [label=above right:2] {};
    	\node[pt] (5) at (0,4) [label=above right:0] {};
    	\node[pt,blue] (6) at (0,5) [label=above right:1] {};
    	\draw (1) -- (2) (5) -- (6);
    	\draw[dashed] (2) -- (3) -- (4) -- (5);
    	\draw[dashed] (1) -- +(1,0) (2) -- +(1,0) (3) -- +(1,0) (4) -- +(1,0) (5) -- +(1,0) (6) -- +(1,0);
    \end{scope}
    \begin{scope}[shift={(2,0)}]
    	\node[pt,blue] (1) at (0,0) [label=above right:7] {};
    	\node[pt] (2) at (0,1) [label=above right:6] {};
    	\node[pt] (3) at (0,2) [label=above right:4] {};
    	\node[pt] (4) at (0,3) [label=above right:2] {};
    	\node[pt] (5) at (0,4) [label=above right:0] {};
    	\node[pt,blue] (6) at (0,5) [label=above right:1] {};
    	\draw (1) -- (2) (5) -- (6);
    	\draw[dashed] (2) -- (3) -- (4) -- (5);
    	\draw[dashed] (1) -- +(1,0) (2) -- +(1,0) (3) -- +(1,0) (4) -- +(1,0) (5) -- +(1,0) (6) -- +(1,0);
    \end{scope}
    \begin{scope}[shift={(3,0)}]
    	\node[pt,blue] (1) at (0,0) [label=above right:9] {};
    	\node[pt] (2) at (0,1) [label=above right:7] {};
    	\node[pt] (3) at (0,2) [label=above right:5] {};
    	\node[pt] (4) at (0,3) [label=above right:3] {};
    	\node[pt] (5) at (0,4) [label=above right:1] {};
    	\node[pt,blue] (6) at (0,5) [label=above right:8] {};
    	\draw (1) -- (2) (5) -- (6);
    	\draw[dashed] (2) -- (3) -- (4) -- (5);
    	\draw[dashed] (1) -- +(1,0) (2) -- +(1,0) (3) -- +(1,0) (4) -- +(1,0) (5) -- +(1,0) (6) -- +(1,0);
    \end{scope}
    \begin{scope}[shift={(4,0)}]
    	\node[pt,blue] (1) at (0,0) [label=above right:9] {};
    	\node[pt] (2) at (0,1) [label=above right:7] {};
    	\node[pt] (3) at (0,2) [label=above right:5] {};
    	\node[pt] (4) at (0,3) [label=above right:3] {};
    	\node[pt] (5) at (0,4) [label=above right:1] {};
    	\node[pt,blue] (6) at (0,5) [label=above right:8] {};
    	\draw (1) -- (2) (5) -- (6);
    	\draw[dashed] (2) -- (3) -- (4) -- (5);
    	\draw[dashed] (1) -- +(1,0) (2) -- +(1,0) (3) -- +(1,0) (4) -- +(1,0) (5) -- +(1,0) (6) -- +(1,0);
    \end{scope}
    \begin{scope}[shift={(5,0)}]
    	\node[pt,blue] (1) at (0,0) [label=above right:6] {};
    	\node[pt] (2) at (0,1) [label=above right:9] {};
    	\node[pt] (3) at (0,2) [label=above right:9] {};
    	\node[pt] (4) at (0,3) [label=above right:8] {};
    	\node[pt] (5) at (0,4) [label=above right:8] {};
    	\node[pt,blue] (6) at (0,5) [label=above right:0] {};
    	\draw (1) -- (2) (5) -- (6);
    	\draw[dashed] (2) -- (3) -- (4) -- (5);
    	\draw (1) -- +(1,0) (2) -- +(1,0) (3) -- +(1,0) (4) -- +(1,0) (5) -- +(1,0) (6) -- +(1,0);
    \end{scope}
    \begin{scope}[shift={(6,0)}]
    	\node[pt,red] (1) at (0,0) [label=above right:9] {};
    	\node[pt,blue] (2) at (0,1) [label=above right:4] {};
    	\node[pt,blue] (3) at (0,2) [label=above right:4] {};
    	\node[pt,blue] (4) at (0,3) [label=above right:2] {};
    	\node[pt,blue] (5) at (0,4) [label=above right:2] {};
    	\node[pt,red] (6) at (0,5) [label=above right:8] {};
    	\draw (1) -- (2) (5) -- (6);
    	\draw[dashed] (2) -- (3) -- (4) -- (5);
    \end{scope}
    \end{tikzpicture}    
    \caption{A proper 10-coloring of $(P_6\square P_7)_\delta$. The vertices in $V_2$, $V_3$ and $V_4$ are shown in red, blue and black, respectively. The vertices in each $V_i$ for $i=2,3,4$ are pairwise adjacent in $(P_6\square P_7)_\delta$ except for the pairs with a dashed line.}
    \label{fig:path-path}
\end{figure}

\section{Conclusion}
We give a structure of $(G_1\square\cdots\square G_k )_\delta$ associated with $(G_1)_{\delta}\square\cdots\square (G_k)_{\delta}$ and the necessary and sufficient condition that both graphs are equal.
We also give sharp bounds on the $\delta$-chromatic number of $G\square H$ with a class of graphs achieving such bound. The $\delta$-chromatic number of the Cartesian product of several classes of well-known graphs are also given.

\section*{Acknowledgement}
The authors thank Rasimate Maungchang for his invaluable comments.
This research project was financially supported by Mahasarakham University.

\bibliography{ref}
\bibliographystyle{plain}
\end{document}